\documentclass[12pt]{article}
\usepackage{e-jc}
\specs{}{}{}

\usepackage{amsthm,amsmath,amssymb}
\usepackage{mathrsfs}

\usepackage{graphicx}

\usepackage[colorlinks=true,citecolor=black,linkcolor=black,urlcolor=blue]{hyperref}

\theoremstyle{plain}
\newtheorem{theorem}{Theorem}
\newtheorem{lemma}{Lemma}
\newtheorem{corollary}{Corollary}
\newtheorem{proposition}{Proposition}

\theoremstyle{definition}
\newtheorem{definition}{Definition}
\newtheorem{example}{Example}
\newtheorem{conjecture}{Conjecture}

\theoremstyle{remark}
\newtheorem{remark}[theorem]{Remark}

\numberwithin{equation}{section}

\newcommand{\abs}[1]{\left\vert#1\right\vert}

\newcommand{\A}{\mathfrak A}
\newcommand{\I}{\mathscr{I}}
\newcommand{\Text}[1]{\text{\textnormal{#1}}}

\title{\bf Parity Types, Cycle Structures \\ and Autotopisms of Latin Squares}

\author{Daniel Kotlar\thanks{The author thanks Ian Wanless for providing a large part of the code needed to perform the computer simulations, and an anonymous referee for many helpful comments.}\\
\small Department of Computer Science\\[-0.8ex]
\small Tel-Hai College\\[-0.8ex]
\small Upper Galilee 12210, Israel\\
\small\tt dannykot@telhai.ac.il\\
}

\date{\dateline{Jan 10, 2012}{Jul 10, 2012}}

\begin{document}
\maketitle

\begin{abstract}
  The parity type of a Latin square is defined in terms of the numbers of even and odd rows and columns. It is related to an Alon-Tarsi-like conjecture that applies to Latin squares of odd order. Parity types are used to derive upper bounds for the size of autotopy groups. A new algorithm for finding the autotopy group of a Latin square, based on the cycle decomposition of its rows, is presented, and upper bounds for the size of autotopy groups are derived from it.

\end{abstract}

\section{Introduction}\label{sec1}
For a given positive integer $n$ let $[n]$ denote the set $\{1,\ldots,n\}$.
A \emph{Latin square} of order $n$ is an $n\times n$ array of numbers in $[n]$ so that each number appears exactly once in each row and column.
Let $LS(n)$ be the set of Latin squares of order $n$ and let $S_n$ be the symmetric group of permutations of $[n]$.
The group $S_n^3=S_n\times S_n \times S_n$ acts on $LS(n)$ by \emph{isotopism}.
An isotopism is a triple $(\alpha,\beta,\gamma)\in S_n^3$ that acts on $LS(n)$, so that $\alpha$ permutes the rows, $\beta$ permutes the columns and $\gamma$ permutes the symbols of a given Latin square.
Two Latin squares are called \emph{isotopic} if there is an isotopism that transforms one to the other.
Being isotopic is an equivalence relation and the set of Latin squares that are isotopic to each other is an \emph{isotopy class}.
Let $\I_n=S_n^3$ be the group of isotopisms on $LS(n)$, $\I_n(L)$ the isotopy class of $L$, and $\A(L)=\{\Theta\in \I_n|\Theta(L)=L\}$ the \emph{autotopy group} of $L$ (also denoted in the literature by $\Text{Is}(L)$ \cite{MckayWan05}, $\Text{Atp}(L)$ \cite{Bro2012,StWan12}, and $\Text{Atop}(L)$ \cite{Stones2010}), whose members are called \emph{autotopisms} of $L$.
We have that $|\I_n|=(n!)^3$ and $\abs{\I_n(L)}=(n!)^3/\abs{\A(L)}$.
\\
The rows and columns of a Latin square can be viewed as permutations of the ordered set $(1,2,\ldots,n)$. A \emph{reduced} (also called \emph{normalized}) Latin square is a Latin square whose first row and first column are the identity permutation. Let $RLS(n)$ be the set of reduced Latin squares of order $n$. For $L\in RLS(n)$ let $\I_{n,L}^\prime=\{\Theta\in\mathscr{I}_n|\Theta(L)\in RLS(n)\}$ and let $\I_n^\prime(L)=\I_n(L)\cap RLS(n)$. Any Latin square can be transformed by isotopism to a reduced Latin square by permuting the rows to form the identity permutation in the first column and then permuting the columns $2,\ldots,n$ to form the identity permutation in the first row. Thus, each reduced Latin square is isotopic to $n!(n-1)!$ distinct Latin squares in this way. It follows that $\abs{\I_n(L)}/\abs{\I_n^\prime(L)}=n!(n-1)!$ and hence $\abs{\I_{n,L}^\prime}=n\cdot n!$ and
\begin{equation}\label{eq1:2}
\abs{\I_n^\prime(L)}=n\cdot n!/\abs{\A(L)}
\end{equation}
For these and other facts about isotopisms and autotopisms the reader is referred to D\'{e}nes and A. D. Keedwell's book \cite{Denes74} and to the works of Janssen \cite{Janssen95}, McKay and Wanless \cite{MckayWan05} and McKay, Meynert and Myrvold \cite{MckayMM07} among many others.
\\
The \emph{sign} (or \emph{parity}) function $\Text{sgn}:S_n\rightarrow \{-1,1\}$ is defined so that $\Text{sgn}(\sigma)=1$ if and only if $\sigma$ is an even permutation. The \emph{parity} of a Latin square $L$, denoted $\Text{par}(L)$, is the product of the signs of all its rows and columns. A Latin square is \emph{even} (resp. \emph{odd}) if its parity is 1 (resp. -1). Let $ELS(n)$ (resp. $OLS(n)$) be the set of even (resp. odd) Latin squares of order $n$. Let $RELS(n)$ and $ROLS(n)$ be the corresponding sets of reduced Latin squares. If $n$ is odd $\abs{ELS(n)}=\abs{OLS(n)}$. If $n$ is even, a known conjecture of Alon and Tarsi \cite{AlonTarsi92} states that $\abs{ELS(n)}\neq\abs{OLS(n)}$. It was also conjectured in \cite{StWan12} that $\abs{RELS(n)}\neq\abs{ROLS(n)}$ for all $n$ (for even $n$ this is equivalent to the Alon-Tarsi conjecture). For further results regarding even and odd Latin squares, the reader is referred to the works of Janssen \cite{Janssen95}, Drisko \cite{Drisko97, Drisko98}, Zappa \cite{Zappa97}, Glynn \cite{Glynn10} and Stones and Wanless \cite{StWan12}.
\\
This work is mainly concerned with the relation between parities of Latin squares and of their rows and columns and autotopy groups. In Section~\ref{sec2} an expression for the members of $\I_{n,L}^\prime$ is derived and the notions of \emph{parity type} and \emph{parity class} are introduced. These lead to expressions for the numbers of even and odd reduced Latin squares of odd order in an isotopy class. These expressions, although stated and proved differently coincide with results of Stones and Wanless \cite{StWan12}. In Section~\ref{sec3} an extension of the Alon-Tarsi Latin square conjecture is discussed and a related conjecture on parity types is presented. In Sections~\ref{sec5} and \ref{sec6} bounds are obtained for the order of the autotopy group $\A(L)$. In Section~\ref{sec5} the bounds are obtained using parities of the rows and columns of $L$ and in Section~\ref{sec6} the bounds are derived from the cycle structure of the rows of $L$.
\\
\textbf{Convention:} Throughout the manuscript, when viewing rows and columns of a Latin square as permutations in $S_n$, the following convention will hold: the number $i$ appearing in the $j$th place of the row (column) $\sigma$ of $L$ signifies that $\sigma(i)=j$. By this convention, in order to transform a column (or row) $\sigma$ to the identity permutation, we have to apply the permutation $\sigma^{-1}$ to the rows (resp. columns) of $L$ (see Lemma~\ref{lem1}(i)). It is also assumed that permutations are applied from right to left.
\section{Parity types and  parity classes}\label{sec2}
A permutation acting on a Latin square acts on another permutation (a row or a column) in two different ways, as described by the following lemma. Recall the convention of this paper that when $i$ appears in the $j$th place of a row (or column) $\pi$ it means that $\pi(i)=j$.
\begin{lemma}\label{lem1}
Let $L$ be a Latin square of order $n$.
\begin{enumerate}
\item[\Text{(i)}] The result of permuting the rows (resp. columns) of $L$, by $\alpha\in S_n$, on any column (resp. row) $\pi$ is $\alpha\pi$.
\item[\Text{(ii)}] The result of permuting the symbols of $L$, by $\alpha\in S_n$, on any row or column $\pi$ is $\pi\alpha^{-1}$.
\end{enumerate}
\end{lemma}
\begin{proof}
(i) Let $\pi$ be a row or a column. Suppose $\pi(i)=j$ and $\alpha(j)=k$ then $i$ appears in the $j$th place of $\pi$, and after applying $\alpha$, the number in the $j$th place of $\pi$ ($i$ in this case) moves to the $k$th place. Thus the resulting permutation takes $i$ to $k$, so it is $\alpha\pi$.
\\
(ii) Denote by $\alpha(\pi)$ the permutation obtained by applying $\alpha$ to the each symbol In the row (or column) $\pi$. Suppose $\pi(i)=j$ and $\alpha(i)=k$. Since $i$ is in the $j$th place of $\pi$, $k$ is in the $\j$th place of $\alpha(\pi)$. Thus $\alpha(\pi)(k)=j$. It follows that $\alpha(\pi)=\pi\alpha^{-1}$.
\end{proof}
The following proposition appears in a different form as part of the proof of Theorem 2.1 in \cite{StWan12}. It describes the $n\cdot n!$ elements in the set $\I_{n,L}^\prime$.
\begin{proposition}\label{prop2:1}
Let $L$ be a reduced Latin square. For $j=1,\ldots,n$ let $\sigma_j$ (resp. $\pi_j$) be the $j$th row (resp. column) of $L$. For any permutation $\alpha\in S_n$ and any $j=1,\ldots,n$, $(\alpha, \alpha\pi_j\sigma_{\alpha^{-1}(1)}^{-1}, \alpha\pi_j)\in \I_{n,L}^\prime$ and any element of $\I_{n,L}^\prime$ is of this form.
\end{proposition}
\begin{proof}
Let $\Theta=(\alpha,\beta,\gamma)$ be such that $\Theta(L)\in RLS(n)$. Since $\gamma$ is applied to all the symbols, the first row and first column of $(\alpha,\beta,1)(L)$ must be identical. Hence, after permuting the rows of $L$ by $\alpha$, $\beta$ needs to permute the columns of $(\alpha, 1, 1)(L)$ to yield a first row that is identical to one of the columns of $(\alpha, 1, 1)(L)$. Note that the first row of  $(\alpha, 1, 1)(L)$ is $\sigma_{\alpha^{-1}(1)}$ and the $j$th column of $(\alpha, 1, 1)(L)$ is $\alpha\pi_j$ (Lemma~\ref{lem1}(i)). Thus, after permuting the columns of $(\alpha, 1, 1)(L)$ by $\beta=\alpha\pi_j\sigma_{\alpha^{-1}(1)}^{-1}$ the first row of the resulting square is $\alpha\pi_j$ and, since $\alpha\pi_j$ is also one of the columns (it is was the $j$th column before the columns were permuted), it must be the first column of $(\alpha, \beta, 1)(L)$. Now, $\gamma$ needs to permute the symbols of $\alpha\pi_j$ to obtain the identity permutation (a reduced Latin square). By Lemma~\ref{lem1}(ii) $\alpha\pi_j\gamma^{-1}=1$ and thus $\gamma=\alpha\pi_j$.
\end{proof}
When $n$ is even, applying an isotopism preserves the parity of a Latin square (\cite{Drisko97}, Lemma 1), and thus the parity of all the Latin squares in an isotopy class is the same. When $n$ is odd an isotopy class contains an equal number of even and odd Latin squares. This is not the case when only reduced Latin squares are concerned. In order to determine the number of even and odd reduced Latin squares in an isotopy class the following definition will be useful:
\begin{definition}\label{def2:1}
The \emph{parity type} of a Latin square $L$ of order $n$ is defined to be $(k,m)$, for $0\leq k, m\leq n/2$, if $k$ (respectively $m$) of its rows (resp. columns) have one sign and the remaining $n-k$ rows (resp. $n-m$ columns) have the opposite sign. The \emph{parity class} $(k,m)$ consists of all Latin squares of parity type $(k,m)$.
\end{definition}
\begin{proposition}\label{prop2:5}
A parity class is the union of isotopy classes.
\end{proposition}
\begin{proof}
Since an isotopy either changes the signs of all the rows (columns) or leaves them all unchanged, the parity type of two isotopic Latin squares is equal.
\end{proof}
Thus, it is valid to refer to the parity type of an isotopy class, as parity type of any of its members.
\begin{lemma}\label{lem2:2}
Suppose $n$ is odd. Let $L\in RLS(n)$ and let $\Theta=(\alpha, \alpha\pi_j\sigma_{\alpha^{-1}(1)}^{-1}, \alpha\pi_j)\in \I_{n,L}^\prime$, as in Proposition~\ref{prop2:1}, then
\begin{equation*}
\Text{par}(\Theta(L))=\Text{sgn}(\sigma_{\alpha^{-1}(1)})\Text{sgn}(\pi_j)\Text{par}(L).
\end{equation*}
\end{lemma}
\begin{proof}
By Proposition~3.1 in \cite{Janssen95}, $\Text{par}(\Theta(L))=(\Text{sgn}(\alpha))^n(\Text{sgn}(\alpha\pi_j\sigma_{\alpha^{-1}(1)}^{-1}))^n\Text{par}(L)$. Since $n$ is odd, $\Text{sgn}(\sigma_{\alpha^{-1}(1)}^{-1})=\Text{sgn}(\sigma_{\alpha^{-1}(1)})$ and $\Text{sgn}(\alpha)$ appears an even number of times, the result follows.
\end{proof}
The following theorem appears in a different form in \cite{StWan12} (Theorem 2.1). The proof here is different and was obtained independently.
\begin{theorem}\label{thm1}
Suppose $n$ is odd and the parity type of $L\in RLS(n)$ is $(k,m)$, then
\begin{equation}\label{even_squares}
\abs{\I_{n}^\prime(L)\cap RELS(n)}=\left\{\begin{array}{lll}
\frac{km+(n-k)(n-m)}{n^2}\abs{\I_{n}^\prime(L)} & \Text{if}& k\equiv m\; (\Text{mod}\; 2)\vspace{.1in}\\
\frac{k(n-m)+m(n-k)}{n^2}\abs{\I_{n}^\prime(L)} & \Text{if}& k\not\equiv m\; (\Text{mod}\; 2).
\end{array}\right.
\end{equation}
\end{theorem}
\begin{proof}
Suppose $k$ and $m$ are either both even or both odd (first row in (\ref{even_squares}). Since $n$ is odd, $n-k$ and $n-m$ are either both odd or both even respectively. Let $I=\{i_1,\ldots,i_k\}$ be the set of indices of the $k$ rows of $L$ having the same sign and let $J=\{j_1,\ldots,j_m\}$ be the set of indices of the $m$ columns of $L$ having the same sign. Suppose $L$ is even. Since $k$ and $m$ have the same parity, the rows indexed by $I$ and the columns indexed by $J$ must have the same sign and the rest of the rows and columns all have the opposite sign. Now, let $\Theta=(\alpha, \alpha\pi_j\sigma_{\alpha^{-1}(1)}^{-1}, \alpha\pi_j)\in \I_{n,L}^\prime$ (see Proposition~\ref{prop2:1}) be such that $\Theta(L)$ is even. By Lemma~\ref{lem2:2}, $\sigma_{\alpha^{-1}(1)}$ and $\pi_j$ must have the same sign. So, either $\alpha^{-1}(1)\in I$ and $j\in J$ or $\alpha^{-1}(1)\in [n]\setminus I$ and $j\in [n]\setminus J$. Since the number of permutations $\alpha$ such that $\alpha^{-1}(1)\in I$ (resp. $\alpha^{-1}(1)\not\in I$) is $(k/n)n!=k(n-1)!$ (resp. $(n-k)(n-1)!$), the number of isotopies $\Theta$ such that $\Theta(L)$ is even is $(km+(n-k)(n-m))(n-1)!$ and thus the number of distinct reduced even Latin squares that are isotopic to $L$ is $[km+(n-k)(n-m)](n-1)!/\abs{\A(L)}$. Hence,
\begin{equation*}
\begin{split}
\abs{\I_{n}^\prime(L)\cap RELS(n)}&=\frac{[km+(n-k)(n-m)](n-1)!}{\abs{\A(L)}}\\
&=\frac{[km+(n-k)(n-m)](n-1)!}{n\cdot n!}\frac{n\cdot n!}{\abs{\A(L)}}\\
&=\frac{[km+(n-k)(n-m)]}{n^2}\abs{\I_{n}^\prime(L)},
\end{split}
\end{equation*}
by (\ref{eq1:2}). The other case is proved in an analogous manner and is left to the reader.
\end{proof}
\begin{remark}
For $\abs{\I_{n}^\prime(L)\cap ROLS(n)}$ the cases in (\ref{even_squares}) are reversed since $km+(n-k)(n-m)+k(n-m)+m(n-k)=n^2$.
\end{remark}
\begin{remark}
Note that if $k=m=0$ all the squares in $\I_{n}^\prime(L)$ are even.
\end{remark}
\section{On even and odd Latin squares of odd order}\label{sec3}
The following extension to all $n$ of the Alon-Tarsi Latin square conjecture \cite{AlonTarsi92} appears in \cite{StWan12}:
\begin{conjecture}\label{conjKot}
For all $n$, $RELS(n)-ROLS(n)\neq0$.
\end{conjecture}
When $n$ is even Conjecture~\ref{conjKot} is equivalent to the Alon-Tarsi conjecture. When $n$ is odd it is not known whether Conjecture~\ref{conjKot} is equivalent to another extension of the Alon-Tarsi conjecture by Zappa \cite{Zappa97}.
\\
Let $N_n(k,m)$ be the number of reduced Latin squares of order $n$ in the parity class $(k,m)$. We have the following corollary of Theorem~\ref{thm1}:
\begin{corollary}\label{cor3:2}
Suppose $n$ is odd. The parity class $(k,m)$ contains $\frac{km+(n-k)(n-m)}{n^2}N_n(k,m)$ even reduced Latin squares if $k\equiv m\; (\Text{mod}\; 2)$ and $\frac{k(n-m)+m(n-k)}{n^2}N_n(k,m)$ even reduced Latin squares if $k\not\equiv m\; (\Text{mod}\; 2)$
\end{corollary}
\begin{proof}
Since a parity class is the union of isotopy classes (Proposition~\ref{prop2:5}), the result follows from Theorem~\ref{thm1} by summing over all the isotopy classes contained in a given parity class.
\end{proof}
Corollary~\ref{cor3:2} yields an equivalent version of Conjecture~\ref{conjKot} for odd $n$:
\begin{proposition}
If $n$ is odd Conjecture~\ref{conjKot} is equivalent to
\begin{equation}\label{eq:prop3}
\sum_{k,m<n/2}(-1)^{k+m}(n-2k)(n-2m)N_n(k,m)\neq0.
\end{equation}
\end{proposition}
\begin{proof}
By Corollary~\ref{cor3:2}, if $k\equiv m\; (\Text{mod}\; 2)$ the difference between the number of even reduced Latin squares and the number of odd reduced Latin squares in the parity class $(k,m)$ is
\begin{equation*}
\left(\frac{km+(n-k)(n-m)}{n^2}-\frac{k(n-m)+m(n-k)}{n^2}\right)N_n(k,m)=\frac{(n-2k)(n-2m)}{n^2}N_n(k,m)
\end{equation*}
and when $k\not\equiv m\; (\Text{mod}\; 2)$, this difference is $-\frac{(n-2k)(n-2m)}{n^2}N_n(k,m)$. Summing over all parity classes,
\begin{equation}\label{sum:classes}
RELS(n)-ROLS(n)=\frac{1}{n^2}\sum_{k,m<n/2}(-1)^{k+m}(n-2k)(n-2m)N_n(k,m)
\end{equation}
Assuming Conjecture~\ref{conjKot} statement (\ref{eq:prop3}) follows.
\end{proof}

This calls for a study of the values $N_n(k,m)$. When $n=5$ only the parity classes $(0,0)$ and $(1,1)$ are nonempty. When $n=6$ the only empty classes are $(0,2)$ and $(2,0)$. When $n=7$ all parity classes are nonempty, as is the case for $n=8,9,10$ and 11 (Tables \ref{tab1}-\ref{tab5}). Thus, it seems reasonable to conjecture that for all $n\geq7$ and all $k,m\leq n/2$ the parity class $(k,m)$ is nonempty. Furthermore, an estimate for the values $N_n(k,m)$ is desired. It is conjectured in \cite{Cav08} that the distribution of the permutation that maps between two randomly chosen rows of a randomly chosen Latin square of order $n$, among all derangements, approaches uniformity as $n\rightarrow\infty$ (a derangement is a permutation without fixed points). Based on this, it seems reasonable to conjecture that for large enough $n$ the distribution of the parities of the rows and columns of a random Latin square will be close to the random distribution. This should also hold when restricting to reduced Latin squares, namely, if the parities of the rows and columns were chosen randomly, the proportion of reduced Latin squares of parity class $(k,m)$ among all reduced Latin squares would be $rs\binom{n}{k}\binom{n}{m}/2^{2n}$, where $r=2$ (resp. $s=2$) if $k<n/2$ (resp. $m<n/2$) and 1 otherwise. (since $k,m\leq n/2$). Thus, $rs$ is always 4 if $n$ is odd.
\begin{conjecture}\label{conj3}
If $n$ is odd, then for all $k,m<n/2$
\begin{equation}\label{conj3:1}
N_n(k,m) \sim\frac{\binom{n}{k}\binom{n}{m}}{2^{2(n-1)}}|RLS(n)|,
\end{equation}
and if $k$ is even, then for all $k,m\leq n/2$
\begin{equation}\label{conj3:2}
N_n(k,m) \sim\frac{rs\binom{n}{k}\binom{n}{m}}{2^{2n}}|RLS(n)|
\end{equation}
where $r=1$ (resp. $s=1$) if $k=n/2$ (resp. $m=n/2$) and $2$ otherwise.
\end{conjecture}

\begin{table}[htb]
\centering 
\begin{tabular}{c c c c}
\hline\hline 
$(k,m)$ & $N_7(k,m)$ & Expected $N_7(k,m)$ &$\frac{N_7(k,m)}{\Text{Expected}\;N_7(k,m)}$\\ [1.0ex] 
\hline
$(0,0)$ & 3270 & 4136 & 0.79\\
$(1,0)+(0,1)$ & 47040 & 57904 & 0.81\\
$(2,0)+(0,2)$ & 246960 & 173720 & 1.42\\
$(3,0)+(0,3)$ & 208740 & 289536 & 0.72\\
$(1,1)$ & 345450 & 202676 & 1.70\\
$(2,1)+(1,2)$ & 979020 & 1216056 & 0.81\\
$(3,1)+(1,3)$ & 2493120 & 2026760 & 1.23\\
$(2,2)$ & 1803690 & 1824084 & 0.99\\
$(3,2)+(2,3)$ & 5538960 & 6080280 & 0.91\\
$(3,3)$ & 5275830 & 5066904 & 1.04\\ [1ex] 
\hline 
\end{tabular}
\caption{Real and expected values of $N_7(k,m)$.} 
\label{tab1}
\end{table}

\begin{table}
\centering 
\begin{tabular}{c c c c}
\hline\hline 
$(k,m)$ & $N_8(k,m)$ & Expected $N_8(k,m)$ &$\frac{N_8(k,m)}{\Text{Expected}\;N_8(k,m)}$\\ [1.0ex] 
\hline
$(0,0)$ & 11 & 8 & 1.38\\
$(0,1)+(1,0)$ & 94 & 96 & 0.98\\
$(0,2)+(2,0)$ & 344 & 344 & 1.00\\
$(0,4)+(4,0)$ & 421 & 428 & 0.98\\
$(0,3)+(3,0)$ & 705 & 680 & 1.04\\
$(1,1)$ & 394 & 392 & 1.01\\
$(1,2)+(2,1)$ & 2604 & 2736 & 0.95\\
$(1,4)+(4,1)$ & 3430 & 3416 & 1.00\\
$(1,3)+(3,1)$ & 5258 & 5472 & 0.96\\
$(2,2)$ & 5085 & 4784 & 1.06\\
$(2,4)+(4,2)$ & 12123 & 11964 & 1.01\\
$(4,4)$ & 7819 & 7477 & 1.05\\
$(2,3)+(3,2)$ & 19000 & 19144 & 0.99\\
$(3,4)+(4,3)$ & 23697 & 23924 & 0.99\\
$(3,3)$ & 17015 & 19140 & 0.99\\ [1ex] 
\hline 
\end{tabular}
\caption{Observed and expected values of $N_8(k,m)$ for a sample of 100,000 randomly chosen reduced Latin squares.}
\label{tab2}
\end{table}

Conjecture~\ref{conj3} is supported by the data in the following tables. Table~\ref{tab1} lists the real values of $N_7(k,m)$, along with the expected values given by (\ref{conj3:1}). The fourth column shows the ratios between the corresponding real and expected values. Tables~\ref{tab2}-\ref{tab5} list the counts of $N_n(k,m)$, as calculated for 100,000 randomly selected reduced Latin squares (using the Jacobson-Matthews method \cite{JM96}), for each of $n=8,9,10$ and 11, along with the expected values calculated using (\ref{conj3:1}) and (\ref{conj3:2}) for this sample size (similar simulations were performed for $n=12,\ldots,20$ and close values were observed as well for estimated and observed counts). The counts for the classes $(k,m)$ and $(m,k)$, when $k\neq m$, were added, since the corresponding classes have equal size. Note that 0 as an estimated value does not mean that the class is expected to be empty, but that for the given sample size no member of the class is expected to be observed.

\begin{table}[!b]
\centering 
\begin{tabular}{c c c c}
\hline\hline 
$(k,m)$ & $N_9(k,m)$ & Expected $N_9(k,m)$ &$\frac{N_9(k,m)}{\Text{Expected}\;N_9(k,m)}$\\ [1.0ex]
\hline
$(0,0)$ & 2 & 0 & -\\
$(0,1)+(1,0)$ & 20 & 24 & 0.83\\
$(0,2)+(2,0)$ & 124 & 112 & 1.11\\
$(1,1)$ & 125 & 124 & 1.01\\
$(0,3)+(3,0)$ & 264 & 256 & 1.03\\
$(0,4)+(4,0)$ & 340 & 384 & 0.89\\
$(1,2)+(2,1)$ & 1031 & 992 & 1.04\\
$(1,3)+(3,1)$ & 2312 & 2304 & 1.00\\
$(1,4)+(4,1)$ & 3479 & 3464 & 1.00\\
$(2,2)$ & 1997 & 1976 & 1.01\\
$(2,3)+(3,2)$ & 9182 & 9232 & 0.99\\
$(2,4)+(4,2)$ & 13704 & 13840 & 0.99\\
$(3,3)$ & 10799 & 10768 & 1.00\\
$(3,4)+(4,3)$ & 32468 & 32296 & 1.01\\
$(4,4)$ & 24153 & 24224 & 1.00\\[1ex] 
\hline 
\end{tabular}
\caption{Observed and expected values of $N_9(k,m)$ for a sample of 100,000 randomly chosen reduced Latin squares.}
\label{tab3}
\end{table}

\begin{table}[htb]
\centering 
\begin{tabular}{c c c c}
\hline\hline 
$(k,m)$ & $N_{10}(k,m)$ & Expected $N_{10}(k,m)$ &$\frac{N_{10}(k,m)}{\Text{Expected}\;N_{10}(k,m)}$\\ [1.0ex]
\hline
$(0,0)$ & 2 & 0 & -\\
$(0,1)+(1,0)$ & 8 & 8 & 1.00\\
$(0,2)+(2,0)$ & 32 & 32 & 1.00\\
$(1,1)$ & 40 & 40 & 1.00\\
$(0,3)+(3,0)$ & 86 & 88 & 0.98\\
$(0,5)+(5,0)$ & 107 & 96 & 1.11\\
$(0,4)+(4,0)$ & 165 & 160 & 1.03\\
$(1,2)+(2,1)$ & 335 & 344 & 0.97\\
$(1,3)+(3,1)$ & 881 & 912 & 0.97\\
$(1,5)+(5,1)$ & 939 & 960 & 0.98\\
$(2,2)$ & 769 & 772 & 1.00\\
$(1,4)+(4,1)$ & 1652 & 1600 & 1.03\\
$(2,3)+(3,2)$ & 4060 & 4120 & 0.99\\
$(2,5)+(5,2)$ & 4181 & 4324 & 0.97\\
$(2,4)+(4,2)$ & 7326 & 7208 & 1.02\\
$(3,3)$ & 5400 & 5492 & 0.98\\
$(3,5)+(5,3)$ & 11461 & 11536 & 0.99\\
$(5,5)$ & 6083 & 6056 & 1.00\\
$(3,4)+(4,3)$ & 19208 & 19224 & 1.00\\
$(4,5)+(5,4)$ & 20286 & 20188 & 1.00\\
$(4,4)$ & 16979 & 16824 & 1.01\\[1ex] 
\hline 
\end{tabular}
\caption{Observed and expected values of $N_{10}(k,m)$ for a sample of 100,000 randomly chosen reduced Latin squares.}
\label{tab4}
\end{table}


\begin{table}[htb]
\centering 
\begin{tabular}{c c c c}
\hline\hline 
$(k,m)$ & $N_{11}(k,m)$ & Expected $N_{11}(k,m)$ &$\frac{N_{11}(k,m)}{\Text{Expected}\;N_{11}(k,m)}$\\ [1.0ex]
\hline
$(0,0)$ & 0 & 0 & -\\
$(0,1)+(1,0)$ & 3 & 0 & -\\
$(0,2)+(2,0)$ & 13 & 8 & 1.63\\
$(0,3)+(3,0)$ & 40 & 32 & 1.25\\
$(1,1)$ & 11 & 12 & 0.92\\
$(0,4)+(4,0)$ & 59 & 64 & 0.92\\
$(0,5)+(5,0)$ & 90 & 88 & 1.02\\
$(1,2)+(2,1)$ & 120 & 112 & 1.07\\
$(1,3)+(3,1)$ & 331 & 344 & 0.96\\
$(2,2)$ & 293 & 288 & 1.02\\
$(1,4)+(4,1)$ & 691 & 696 & 0.99\\
$(1,5)+(5,1)$ & 985 & 968 & 1.02\\
$(2,3)+(3,2)$ & 1699 & 1728 & 0.98\\
$(2,4)+(4,2)$ & 3312 & 3464 & 0.96\\
$(2,5)+(5,2)$ & 4881 & 4848 & 1.01\\
$(3,3)$ & 2555 & 2596 & 0.98\\
$(3,4)+(4,3)$ & 10209 & 10385 & 0.98\\
$(3,5)+(5,3)$ & 14446 & 14536 & 0.99\\
$(4,4)$ & 10472 & 10384 & 1.01\\
$(4,5)+(5,4)$ & 29244 & 29080 & 1.01\\
$(5,5)$ & 20546 & 20356 & 1.01\\[1ex] 
\hline 
\end{tabular}
\caption{Observed and expected values of $N_{11}(k,m)$ for a sample of 100,000 randomly chosen reduced Latin squares.}
\label{tab5}
\end{table}
Lemma 2.4 in \cite{StWan12} states that almost all Latin squares belong to parity classes $(k,m)$ such that $k,m\geq n/63$, and thus it can be viewed as a partial result for Conjecture~\ref{conj3}.
\begin{proposition}\label{prop4:3}
Assuming Conjecture~\ref{conj3} holds then
\begin{equation}\label{prop4:3:0}
\lim_{\substack{n\;\Text{odd}\\n\rightarrow\infty}}\frac{|RELS(n)|-|ROLS(n)|}{|RLS(n)|}=0
\end{equation}
\end{proposition}
\begin{proof}
Combining (\ref{sum:classes}) and Conjecture~\ref{conj3} we have that for large odd $n$
\begin{equation*}
\frac{|RELS(n)|-|ROLS(n)|}{|RLS(n)|}\approx\frac{1}{n^22^{2(n-1)}}\sum_{k,m=0}^{\lfloor n/2\rfloor}(-1)^{k+m}(n-2k)(n-2m)\binom{n}{k}\binom{n}{m}
\end{equation*}
Now,
\begin{equation}\label{prop4:3:2}
\begin{split}
\sum_{k,m=0}^{\lfloor n/2\rfloor}(-1)^{k+m}&(n-2k)(n-2m)\binom{n}{k}\binom{n}{m}\\
&=\left(\sum_{k=0}^{\lfloor n/2\rfloor}(-1)^{k}(n-2k)\binom{n}{k}\right)^2\\
&=\left(n\sum_{k=0}^{\lfloor n/2\rfloor}(-1)^{k}\binom{n}{k}
-2\sum_{k=0}^{\lfloor n/2\rfloor}(-1)^{k}k\binom{n}{k}\right)^2.
\end{split}
\end{equation}
It can be shown by induction on $n$ that for odd $n\geq3$
\begin{equation*}
\sum_{k=0}^{\lfloor n/2\rfloor}(-1)^{k}k\binom{n}{k}=(-1)^{\frac{n-1}{2}}n\binom{n-2}{\frac{n-1}{2}}
\end{equation*}
and
\begin{equation*}
\sum_{k=0}^{\lfloor n/2\rfloor}(-1)^{k}\binom{n}{k}=2(-1)^{\frac{n-1}{2}}\binom{n-2}{\frac{n-1}{2}}.
\end{equation*}
Thus, the sum in (\ref{prop4:3:2}) is 0 and the result follows.
\end{proof}
The conjecture presented in Proposition~\ref{prop4:3} is supported by Theorem 3.2 in \cite{StWan12} which states that $|RELS(n)|$ and $|ROLS(n)|$ share a large divisor which grows exponentially as $n\rightarrow\infty$.
\section{On the order of $\A(L)$}\label{sec5}
Suppose $L$ and $L'$ are isotopic Latin squares. Note that if $L'=\Theta(L)$, then for any $\Gamma\in\I_n$, $\Gamma\in\A(L)$ if and only if $\Theta\Gamma\Theta^{-1}\in \A(L')$. Hence $|\A(L)|=|\A(L')|$. Since every Latin square is isotopic to some reduced Latin square, it is enough to study the orders of the autotopy groups of reduced Latin squares.
\\
The parity type of a square provides some information about its autotopy group. Browning, Stones and Wanless (\cite{Bro2012}, Lemma 4.1) showed that for a prime order $p$ there is exactly one isotopism class that contains Latin squares $L$ for which $p^2$ divides $\abs{\A(L)}$, namely, the isotopism class containing the Cayley table of $Z_p$. The following corollary of Theorem~\ref{thm1} extends that result:
\begin{corollary}
Let $L$ be a Latin square of prime order $p>2$ and of parity type $(k,m)$.
\begin{enumerate}
  \item [\Text{(i)}] If $k\neq 0$ or $m\neq0$ then $p^2\not|\abs{\A(L)}$.
  \item [\Text{(ii)}] If $k$ and $m$ are both nonzero then $p\not|\abs{\A(L)}$.
  \item [\Text{(iii)}] If the parity of $L$ is (0,0) but $L$ is not isotopic to the Cayley table of $Z_p$, then $p^2\not|\abs{\A(L)}$

\end{enumerate}

\end{corollary}
\begin{proof}
Statements (i) and (iii) follow from the above mentioned result in \cite{Bro2012} and the fact that the parity type of a Cayley table of odd order is $(0,0)$. Without loss of generality we may assume that $L$ is reduced (otherwise we take a reduced Latin square in the isotopy class of $L$ and use the fact that isotopy preserves the parity type). Setting $n=p$ we look at one of the expressions for $\abs{\I_{p}^\prime(L)\cap RELS(p)}$ in (\ref{even_squares}). If just one of $k$ and $m$ is nonzero, then $p$ divides the numerator but $p^2$ does not. Thus, $p$ must divide $\abs{\I_{p}^\prime(L)}$ and by (\ref{eq1:2}) the highest power of $p$ that may divide $\abs{\A(L)}$ is 1. This proves (ii).
\end{proof}
For the proof of the next theorem the following definition will be useful:
\begin{definition}
Let $I,J\subset[n]$ with $|I|,|J|\leq n/2$. We say that a Latin square $L$ of order $n$ has \emph{parity set} $(I,J)$ if the rows (resp. columns) indexed by $I$ (resp. $J$) are all of the same parity and the rows indexed by $[n]\setminus I$ (resp. $[n]\setminus J$) are all of the opposite parity.
\end{definition}
\begin{remark}\label{rem5:1}
Permuting the columns of a Latin square preserves the first component of its parity set and permuting the rows of a Latin square preserves the second component of its parity set. Permuting the entries of a Latin square preserves its parity set.
\end{remark}
\begin{theorem}\label{thm5:5}
Let $L$ be a Latin square of order $n$ with parity type $(k,m)$, then
\begin{equation}\label{eq5:1}
\abs{\A(L)}\leq \frac{n\cdot n!}{\max\left(\binom{n}{k},\binom{n}{m}\right)}.
\end{equation}
If no column of $L$ is the product of two other columns, then
\begin{equation}\label{eq5:2}
\abs{\A(L)}\leq \frac{n!}{\binom{n}{k}}.
\end{equation}
$($Similarly, if no row of $L$ is the product of two other rows, then $\abs{\A(L)}\leq n!/\binom{n}{m}$.$)$
\end{theorem}
\begin{proof}
 We can assume that $L$ is a reduced Latin square. By Remark~\ref{rem5:1}, choosing all different $\alpha\in S_n$ in $(\alpha, \alpha\pi_j\sigma_{\alpha^{-1}(1)}^{-1}, \alpha\pi_j)\in \I_{n,L}^\prime$ (see Proposition~\ref{prop2:1}) will yield all $\binom{n}{k}$ distinct $I\subset [n]$ of order $k$ as the first component of the parity set of the reduced Latin squares that are isotopic to $L$. Thus $|\I_n^\prime(L)|\geq \binom{n}{k}$. By (\ref{eq1:2}) $\abs{\A(L)}\leq \frac{n\cdot n!}{\binom{n}{k}}$. Since the same argument applies by looking at the second component of the parity set, the inequality (\ref{eq5:1}) follows. Let $\Sigma=\{\alpha_1,\ldots,\alpha_{\binom{n}{k}}\}$ be a set of permutations that, when taken as the first component of corresponding $\binom{n}{k}$ isotopies on $L$, will yield $\binom{n}{k}$ distinct reduced Latin squares with distinct first parity set component. Now, let $\alpha\in\Sigma$ and suppose that $\Theta_1=(\alpha, \alpha\pi_j\sigma_{\alpha^{-1}(1)}^{-1}, \alpha\pi_j)\in \I_{n,L}^\prime$ and $\Theta_2=(\alpha, \alpha\pi_l\sigma_{\alpha^{-1}(1)}^{-1}, \alpha\pi_l)\in \I_{n,L}^\prime$ satisfy $\Theta_1(L)=\Theta_2(L)$. Let $\Phi=\Theta_2^{-1}\Theta_1$ Thus $\Phi\in\A(L)$. That is,
\begin{equation}
\Phi=(1, \sigma_{\alpha^{-1}(1)}\pi_l^{-1}\pi_j\sigma_{\alpha^{-1}(1)}^{-1}, \pi_l^{-1}\pi_j)\in\A(L).
\end{equation}
For simplicity, denote $\sigma=\sigma_{\alpha^{-1}(1)}$. Let $\beta=\sigma\pi_l^{-1}\pi_j\sigma$ and $\gamma=\pi_l^{-1}\pi_j=\sigma^{-1}\beta\sigma$. We have that $\Phi=(1,\beta,\sigma^{-1}\beta\sigma)\in\A(L)$. While applying $\Phi$ to $L$, after permuting the columns by $\beta$, the first row is $\beta$ (by Lemma~\ref{lem1}(i), since it was originally the identity permutation) and we have to apply again $\beta$ on the symbols in order to transform the first row back into the identity permutation (Lemma~\ref{lem1}(ii)). Thus $\gamma=\beta$ and we have $\Phi=(1,\beta,\beta)$. Suppose $\beta(1)=r$. After permuting the columns by $\beta$, the original first column (the identity permutation) moves to the $r$th position. Then when applying $\beta$ on the symbols The $r$th column becomes $\beta^{-1}$ (Lemma~\ref{lem1}(ii)), and since we assumed that $(1,\beta,\beta)\in\A(L)$ we must have that $\beta^{-1}=\pi_r$. But $\beta^{-1}=\gamma^{-1}=\pi_j^{-1}\pi_l$. Thus $\pi_l=\pi_j\pi_r$. If we assume that no column of $L$ is the product of two other columns, we must have that $\Theta_1(L)\neq\Theta_2(L)$. We conclude that for any $\alpha\in\Sigma$ the different $n$ isotopies obtained by taking the $n$ columns of $L$ yield $n$ distinct squares that are isotopic to $L$. Since for different elements of $\Sigma$ the squares obtained have parity sets with distinct first component, $\abs{\I_n^\prime(L)}\geq n\cdot\binom{n}{k}$. By (\ref{eq1:2}) $\abs{\A(L)}\leq \frac{n\cdot n!}{n\cdot\binom{n}{k}}=\frac{ n!}{\binom{n}{k}}$
\end{proof}
\begin{example}
The parity type of the following Latin square of order 6 is $(3,3)$ and thus, by (\ref{eq5:1}), $\abs{\A(L)}\leq (6\cdot6!)/\binom{6}{3}=216$. Indeed, the autotopy group of this square has size 216, so it is possible to reach the bound in (\ref{eq5:1}).
\begin{equation*}
\begin{array}{|c|c|c|c|c|c|}
\hline
1 & 2 & 3 & 4 & 5 & 6\\
\hline
2 & 1 & 4 & 3 & 6 & 5\\
\hline
3 & 5 & 1 & 6 & 2 & 4\\
\hline
4 & 6 & 2 & 5 & 1 & 3\\
\hline
5 & 3 & 6 & 1 & 4 & 2\\
\hline
6 & 4 & 5 & 2 & 3 & 1\\
\hline 
\end{array}
\end{equation*}
\end{example}
\section{Computing $\A(L)$ by cycle structures}\label{sec6}
Every permutation $\alpha\in S_n$ can be decomposed into a unique (up to order) product of disjoint cycles. Cycle structures of permutations were considered in the context of Latin squares in different aspects. Cavenagh, Greenhill and Wanless \cite{Cav08} considered the cycle structure of the permutation that transforms one row of a Latin square to another row. Other works \cite{Fal09, Stones11} considered the cycle structure of the permutations $\alpha$, $\beta$ and $\gamma$ in an isotopism $\Theta=(\alpha, \beta, \gamma)$, in order to derive information on Latin squares for which $\Theta$ is an autotopism. Here the cycle structure of the rows of a Latin square $L$ are considered in order to derive information on $\A(L)$.
\begin{lemma}\label{lem6:1}
Let $L$ be a reduced Latin square of order $n$ with rows $\{\sigma_i\}_{i=1}^n$. Let $\Theta=(\alpha, \beta, \gamma)\in \A(L)$. Then for each $i=1,\ldots,n$, $\sigma_{\alpha^{-1}(1)}^{-1}\sigma_{\alpha^{-1}(i)}$ has the same cycle structure as $\sigma_i$.
\end{lemma}
\begin{proof}
By Proposition~\ref{prop2:1}, $\Theta$ has the form $\Theta=(\alpha, \alpha\pi_j\sigma_{\alpha^{-1}(1)}^{-1}, \alpha\pi_j)\in \A(L)$ for some column $\pi_j$ of $L$. We follow the $i$th row of $L$ as $\Theta$ is applied. Originally it is $\sigma_i$. After permuting the rows the $i$th row is $\sigma_{\alpha^{-1}(i)}$ (following the convention that $i$ in the $j$th place of a permutation $\alpha$ signifies that $\alpha(i)=j$). After permuting the columns by $\alpha\pi_j\sigma_{\alpha^{-1}(1)}^{-1}$ the $i$th row becomes $\alpha\pi_j\sigma_{\alpha^{-1}(1)}^{-1}\sigma_{\alpha^{-1}(i)}$ (Lemma~\ref{lem1}(i)) and after permuting the symbols, the resulting $i$th row is $\alpha\pi_j\sigma_{\alpha^{-1}(1)}^{-1}\sigma_{\alpha^{-1}(i)}(\alpha\pi_j)^{-1}$ (Lemma~\ref{lem1}(ii)). Since $\Theta$ is an autotopism of $L$ we have
\begin{equation}\label{eq6:01}
\sigma_i=\alpha\pi_j\sigma_{\alpha^{-1}(1)}^{-1}\sigma_{\alpha^{-1}(i)}(\alpha\pi_j)^{-1}
\end{equation}
Thus $\sigma_i$ and $\sigma_{\alpha^{-1}(1)}^{-1}\sigma_{\alpha^{-1}(i)}$ are conjugates and hence have the same cycle structure.
\end{proof}
\begin{theorem}\label{thm6:2}
The following is an algorithm for finding $\A(L)$ for a given reduced Latin square $L$ of order $n$.
\begin{enumerate}
\item[\Text{(1)}] Compute the cycle structures of the rows of $L$, viewed as permutations in $S_n$. Let $C_1,\ldots,C_s$ be the distinct cycle structures of the rows, sorted in some well-defined way (see \cite{Fal09}) and let $\lambda=(\lambda_{1},\lambda_{2},\ldots,\lambda_{s})$ be a partition of $n$ where each $\lambda_{i}$ is the number of rows with cycle structure $C_i$.
\item[\Text{(2)}] For each row $\sigma_k$, compute the cycle structures of the set of permutations $\{\sigma_k^{-1}\sigma_i\}_{i=1}^n$. Let $C_1^k,\ldots,C_t^k$ be the cycle structures of these permutations, sorted as in (1), and let $\lambda^k$ be the partition of $n$ corresponding to these cycle structures.
\item[\Text{(3)}] If the ordered sets $\{C_1,\ldots,C_s\}$ and $\{C_1^k,\ldots,C_t^k\}$ coincide and $\lambda=\lambda^k$ construct the set $I_k$ of permutations $\alpha\in S_n$ satisfying:\\
    \Text{(*)} For each $i$, $\sigma_{\alpha(i)}$ and $\sigma_k^{-1}\sigma_i$ have the same cycle structure.
\item[\Text{(4)}] For each $\alpha\in\cup_k I_k$ and each column $\pi_j$, $j=1,\dots,n$, let $\Theta_{\alpha,j}=(\alpha, \alpha\pi_j\sigma_{\alpha^{-1}(1)}^{-1}, \alpha\pi_j)$. Check whether $\Theta_{\alpha,j}(L)=L$. If equality holds then $\Theta_{\alpha,j}\in\A(L)$.
\end{enumerate}
\end{theorem}
\begin{proof}
The condition (*) follows directly from Lemma~\ref{lem6:1}. Since $\alpha$ is a bijection the condition in step (3) must hold.
\end{proof}

\begin{remark}
Condition (*) implies that $\alpha(k)=1$ since the identity permutation has its own unique cycle structure.
\end{remark}
\begin{example}\label{ex1}
Consider the following reduced Latin square $L$ of order 7:
\begin{equation*}
\begin{array}{|c|c|c|c|c|c|c|}
\hline
 1 & 2 & 3 & 4 & 5 & 6 & 7\\
\hline
 2 & 1 & 4 & 3 & 6 & 7 & 5\\
\hline
 3 & 4 & 2 & 5 & 7 & 1 & 6\\
\hline
 4 & 6 & 7 & 1 & 2 & 5 & 3\\
\hline
 5 & 7 & 6 & 2 & 1 & 3 & 4\\
\hline
 6 & 3 & 5 & 7 & 4 & 2 & 1\\
\hline
 7 & 5 & 1 & 6 & 3 & 4 & 2\\
\hline
\end{array}
\end{equation*}
The cycle representations of the rows of $L$, grouped by cycle structure, are:
\begin{equation}\label{eq6:02}
\begin{array}{c l}
\Text{Row} & \Text{Cycle representation}\\
\hline
1   &   (1)(2)(3)(4)(5)(6)(7)\\
\hline
2   &   (1,2)(3,4)(5,7,6)\\
4   &   (1,4)(3,7)(2,5,6)\\
5   &   (1,5)(3,6)(2,4,7)\\
\hline
7   &   (4,6)(1,3,5,2,7)\\
\hline
3   &   (1,6,7,5,4,2,3)\\
6   &   (1,7,4,5,3,2,6)\\
\end{array}
\end{equation}
One set of possible $\alpha$ consists of permutations that satisfy $\alpha(1)=1$, $\alpha(7)=7$ and permute the sets $\{2,4,5\}$ and $\{3,6\}$. This set of permutations contains $3!\cdot2!=12$ permutations. (This corresponds to $k=1$ in Step (2) of the algorithm in Theorem~\ref{thm6:2}.)
\\
If we take the product of each row by $\sigma_4^{-1}$ (the inverse of the 4th row, $k=4$ in Step (2) of the algorithm in Theorem~\ref{thm6:2}) we obtain the same set of cycle structures:
\begin{equation}\label{eq6:03}
\begin{array}{c l}
i & \Text{Cycle rep. of }\sigma_4^{-1}\sigma_i \\
\hline
4   &   (1)(2)(3)(4)(5)(6)(7)\\
\hline
1   &   (1,4)(3,7)(2,6,5)\\
3   &   (1,5)(2,7)(3,4,6)\\
5   &   (1,2)(6,7)(3,5,4)\\
\hline
7   &   (2,3)(1,7,4,5,6)\\
\hline
2   &   (1,6,2,4,7,5,3)\\
6   &   (1,3,6,4,2,5,7)\\
\end{array}
\end{equation}
So, another set of possible $\alpha$'s consists of permutations that satisfy $\alpha(4)=1$, $\alpha(7)=7$ and map the set $\{1,3,5\}$ onto the set $\{2,4,5\}$ and the set $\{2,6\}$ onto the set $\{3,6\}$. Again, there are $3!\cdot2!=12$ such permutations. Taking the product of each row by $\sigma_5^{-1}$ also produces the same set of cycle structures as the original one, while taking the products of the rows with each of $\sigma_k^{-1}$, for $k=2,3,6,7$, produces sets of cycle structures that do not coincide with the original one. Thus, $\abs{\cup_k I_k}=3\cdot12=36$  and there are $36\cdot7=252$ isotopisms to check in step 4, of which only three are autotopisms. Hence $\abs{\A(L)}=3$.
\end{example}

\begin{remark}
The algorithm, as described in Theorem~\ref{thm6:2}, was tested on 100,000 randomly selected Latin squares (Jacobson-Matthews method \cite{JM96}) of each of the orders 8,9,10, and 11. It performed faster than the algorithm by McKay, Meynert and Myrvold \cite{MckayMM07}, based on ``nauty'', but slower than an improved version of ``nauty'' using vertex invariants (such as the ``train'' \cite{wan05}).
\end{remark}
The algorithm of Theorem~\ref{thm6:2} can be massively sped up by constructing the permutations $\alpha$ in Step (4), in parts, corresponding to the cycle structures $C_i^k$ in Step (3). After constructing each part we can perform Step (4) on the rows of $L$ that are permuted by the part of $\alpha$ already constructed. This may rule out most of the candidates, or produce candidates $\alpha$  while saving the time needed to consider all the parts corresponding to the cycle structures $C_i^k$. This idea is illustrated in the following example.
\begin{example}\label{ex2}
Consider the Latin square in Example~\ref{ex1}. We saw that if we take the product of each row with $\sigma_4^{-1}$ the same set of cycle structures as the original one is obtained (see (\ref{eq6:02}) and (\ref{eq6:03})). This implies that one set of possible $\alpha$ may satisfy $\alpha(7)=7$. By (\ref{eq6:01}) it follows that
\begin{equation}\label{eq6:04}
\sigma_7=\alpha\pi_j\sigma_{4}^{-1}\sigma_{7}\pi_j^{-1}\alpha^{-1}
\end{equation}
for one or more columns $\pi_j$. Let $\tau=\sigma_4^{-1}\sigma_7=(2,3)(1,7,4,5,6)$ (cycle decomposition). Since $\pi_1=1$ (the identity permutation) we have $\pi_1 \tau \pi_1^{-1} = \tau$ and by (\ref{eq6:04}) we have
\begin{equation}\label{eq6:05}
    \sigma_7=(4,6)(1,3,5,2,7)=(\alpha(2)\alpha(3))(\alpha(1),\alpha(7),\alpha(4),\alpha(5),\alpha(6)).
\end{equation}
Since we know that $\alpha(4)=1$ and $\alpha(7)=7$, by looking at the second cycle in (\ref{eq6:05}), it follows that $\alpha(5)=3$ and $\alpha(6)=5$, both in contradiction to the constrains in Example~\ref{ex1} requiring that $\alpha$ must map 5 to one of $\{2,4,5\}$ and 6 to one of $\{3,6\}$. Thus all possible $\alpha$ for $k=4$ and $j=1$ in Steps (2)-(4) of Theorem~\ref{thm6:2} are ruled out.

If we take $j=2$ in Step (4) of Theorem~\ref{thm6:2} we obtain $\pi_2 \tau \pi_2^{-1} = (1,6)(2,5,3,7,4)$ and thus, by (\ref{eq6:01}), it follows that
\begin{equation}\label{eq6:06}
    \sigma_7=(4,6)(1,3,5,2,7)=(\alpha(1),\alpha(6))(\alpha(2),\alpha(5),\alpha(3),\alpha(7),\alpha(4)).
\end{equation}
Using the fact that $\alpha(4)=1$ and $\alpha(7)=7$ it follows that $\alpha=(1,4)(2,3)$ is the only permutation that coincides with the constrains in Example~\ref{ex1}. It remains to check whether the isotopy $\Theta_{\alpha,2}(L)=(\alpha, \alpha\pi_2\sigma_{\alpha^{-1}(1)}^{-1}, \alpha\pi_2)$  satisfies $\Theta_{\alpha,2}(L)=L$.

The reader can verify that taking the row $\pi_4$ produces another candidate $\alpha=(1,5,4)(2,6,3)$, while taking the rows $\pi_3,\pi_5,\pi_6$ and $\pi_7$ yield contradictions. Thus, for the case $k=4$ in Step (2), the list of possible autotopisms to check in Step (4) was narrowed down to only 2 in $n=7$ iterations instead of $n\prod_{i=1}^s\lambda_{i}!=7\cdot2!\cdot3!=84$ in $\prod_{i=1}^s\lambda_{i}!=2!\cdot3!=12$ iterations according to the original algorithm.
\end{example}

\begin{corollary}\label{cor3}
Let $L$ be a Latin square of order $n$. Let $C_1,\ldots,C_s$ be the distinct cycle structures of the rows and let $\lambda=(\lambda_{1},\lambda_{2},\ldots,\lambda_{s})$ be a partition of $n$ where each $\lambda_{i}$ is the number of rows with cycle structure $C_i$. Then
\begin{equation}\label{eq6:1}
\abs{\A(L)}\leq n^2\prod_{i=1}^s\lambda_{i}!.
\end{equation}
If no column of $L$ is the product of two other columns then
\begin{equation}\label{eq6:2}
\abs{\A(L)}\leq n\prod_{i=1}^s\lambda_{i}!.
\end{equation}
\end{corollary}
\begin{proof}
Without loss of generality we may assume that $L$ is reduced.
For each $k$ satisfying the condition of step (3) in Theorem~\ref{thm6:2}, there are $\prod_{i=1}^s\lambda_{i}!$ permutations $\alpha$ satisfying (*).
Since there are $n$ rows to check in step (2) and for each $\alpha$ there are $n$ isotopies $\Theta_{\alpha,j}$ to check in step (4) of Theorem~\ref{thm6:2}, Inequality~\ref{eq6:1} follows.
It was shown in the proof of Theorem~\ref{thm5:5} that if two distinct isotpisms, with the same first component $\alpha$, map $L$ to the same reduced Latin square, then $L$ must have a column that is the product of two other columns. If no such column exists, then in step (4) of Theorem~\ref{thm6:2} at most one value of $j$ satisfies $\Theta_{\alpha,j}(L)=L$. Thus Inequality~\ref{eq6:2} follows.
\end{proof}
\begin{remark}\label{rem9}
The worst case in the algorithm in Theorem~\ref{thm6:2} is when $n-1$ rows have the same cycle structure (the first row, the identity permutation, has its own cycle structure). This condition is achieved, for example, by atomic Latin squares \cite{wan05} and by Cayley tables of elementary abelian groups. In this case the bound in (\ref{eq6:1}) is $n\cdot n!$, which provides no information, as $|\I_{n,L}^\prime|=n\cdot n!$.
\end{remark}
\begin{example}
The rows of the Latin square
\begin{equation*}
\begin{array}{|c|c|c|c|c|c|}
\hline
 1 & 2 & 3 & 4 & 5 & 6\\
\hline
 2 & 1 & 4 & 3 & 6 & 5\\
\hline
 3 & 4 & 5 & 6 & 1 & 2\\
\hline
 4 & 3 & 6 & 5 & 2 & 1\\
\hline
 5 & 6 & 1 & 2 & 3 & 4\\
\hline
 6 & 5 & 2 & 1 & 4 & 3\\
\hline
\end{array}
\end{equation*}
have cycle structures
\begin{equation*}
\begin{array}{c l}
\Text{Row} & \Text{Cycle representation}\\
\hline
1   &   (1)(2)(3)(4)(5)(6)\\
\hline
2   &   (1,2)(3,4)(5,6)\\
\hline
3   &   (1,5,3)(2,6,4)\\
5   &   (1,3,5)(2,4,6)\\
\hline
4   &   (1,6,3,2,5,4)\\
6   &   (1,4,5,2,3,6)\\
\end{array}
\end{equation*}
The products of the rows with each $\sigma_k^{-1}$, $k=1,\ldots,6$, produce a set of permutations with the same cycle structures as above. Thus, the set $\cup_k I_k$ has size $2!\cdot2!\cdot6=24$, and the bound in (\ref{eq6:1}) is $24\cdot6=144$. The actual size of $\A(L)$ in this case is 72. This is the closest example found, among Latin squares of order 6 and 7, to the bound in (\ref{eq6:1}).
\end{example}

\bibliographystyle{abbrv}
\bibliography{autotopies_bib}

\end{document}